\documentclass[12pt]{article}
\usepackage{color}
\usepackage{amsfonts}
\usepackage{amsmath}
\usepackage{amsthm} 
\usepackage{epsfig}
\usepackage{graphics}
\usepackage{pslatex}

\usepackage{latexsym}
\usepackage{amssymb,exscale}
\usepackage[all]{xy}
\usepackage{hyperref}
\usepackage{verbatim} 
\usepackage{fullpage} 

\theoremstyle{plain} 
    \newtheorem{theorem}{Theorem}
    \newtheorem{lemma}[theorem]{Lemma}
    \newtheorem{proposition}[theorem]{Proposition}
    \newtheorem{corollary}[theorem]{Corollary}

\theoremstyle{definition} 
    \newtheorem{definition}[theorem]{Definition}

    \newtheorem{remark}[theorem]{Remark}

\def\bb{\mathbb}

\def\<{\langle}
\def\>{\rangle}

\def\bar{\overline}
\def\P{{\bf P}}

\newcommand\tr{{\mbox{\rm tr}}}

\newcommand\floor[1]{\left\lfloor #1 \right\rfloor}
\newcommand\norm[1]{\left\| #1 \right\|}
\newcommand\vnorm[1]{\left\| #1 \right\|_2}
\newcommand\fnorm[1]{\left\| #1 \right\|_2}
\newcommand\abs[1]{\left| #1 \right|}
\newcommand\angles[1]{\left\langle #1 \right\rangle}
\newcommand\braces[1]{\left\{ #1 \right\}}

\newcommand\lt{\left}
\newcommand\rt{\right}
\newcommand\mnote[1]{} 
\newcommand\be{\begin{equation*}}

\newcommand\ee{\end{equation*}}

\newcommand\ben{\begin{equation}}
\newcommand\een{\end{equation}}
\newcommand\bes{\begin{eqnarray*}}
\newcommand\ees{\end{eqnarray*}}

\newcommand{\dist}{\mbox{\rm dist}}
\newcommand{\Span}{\operatorname{Span}}

\newcommand{\sm}{{\raise0.3ex\hbox{$\scriptstyle \setminus$}}}

\renewcommand{\phi}{\varphi}

\renewcommand\Pr[1]{\,\mathbb P\,\lt(\,#1\,\rt)\,}

\def\CHI{\mathchoice%
{\raise2pt\hbox{$\chi$}}%
{\raise2pt\hbox{$\chi$}}%
{\raise1.3pt\hbox{$\scriptstyle\chi$}}%
{\raise0.8pt\hbox{$\scriptscriptstyle\chi$}}}
\def\smalloplus{\raise1pt\hbox{$\,\scriptstyle \oplus\;$}}

\newcommand{\deltane}{\delta_{n,\epsilon}}
\newcommand\vv[1]{\mathbf v_{#1}} 
\newcommand\zbar{\bar z}

\author{
Philip Matchett Wood\thanks{Department of Mathematics, University of Wisconsin-Madison, 480 Lincoln Dr., Madison, WI 53706 \mbox{pmwood@math.wisc.edu}.  This work was partially supported by an NSF postdoctoral grant.}} 

\title{Universality of the ESD for a fixed matrix plus small random noise: a stability approach}

\begin{document}
\date{ \today}
\maketitle

\abstract{
We study the empirical spectral distribution (ESD) in the limit where
$n\to\infty$ of a fixed $n$ by $n$ matrix $M_n$ plus small random noise of the
form $f(n) X_n$, where $X_n$ has iid mean 0, variance $1/n$ entries and
$f(n)\to 0$.  It is known for certain $M_n$, in the case where $X_n$ is iid complex Gaussian, that the limiting distribution of the ESD of $M_n+f(n)X_n$ can be dramatically different from that for $M_n$.  We prove a general universality result showing, with some conditions on $M_n$ and $f(n)$, that the limiting distribution of the ESD does not depend on the type of distribution used for the random entries of $X_n$.  We use the universality result to exactly compute the limiting ESD for two families where it was not previously known.  The proof of the main result incorporates the Tao-Vu replacement principle and a version of the Lindeberg replacement strategy, along with the newly-defined notion of 
stability of sets of rows of a matrix. 
}

\section{Introduction}

Given an $n$ by $n$ complex matrix $A$, we define the \emph{empirical
spectral distribution} (which we will abbreviate \emph{ESD}), to be the
following discrete probability measure on $\bb C$: 
$$\mu_{A}(z):= 
\frac 1n \sum_{j=1}^n \delta_{\lambda_j},
$$ 
where $\lambda_1,\lambda_2,\ldots,\lambda_n$
are the eigenvalues of $A$ with multiplicity and $\delta_x$ is the Dirac
measure centered at $x$.  For a sequence of random matrices $A_n$, we say that $\mu_{A_n}$ \emph{converges in probability} to another probability measure $\mu$ if for every smooth, compactly supported test function $g:\bb C \to \bb C$ we have that $\int_{\bb C} g \, d\mu_{A_n}$ converges in probability to $\int_{\bb C} g\, d \mu$. 

Questions about the limiting distribution of the ESD of random matrices started in the 1950s and have generated much recent interest.  The Circular Law
states that the ESD of a random $n$ by $n$ matrix $X_n$ with iid mean 0, variance $1/n$
entries converges to the uniform measure on the unit disk (see, for
example, \cite{BChafai2012,TaoRMBook2012,TVuK2008} and references therein).
Low rank perturbations of random matrices with iid entries do not change the limiting bulk ESD, even for perturbations up to rank $o(n)$ (see \cite[Corollary~1.12]{TVuK2008}, \cite{Chafai2010}, and \cite{Bordenave2011}); however, such perturbations can produce outlier eigenvalues---see \cite{Tao2013, Tao2013a}.

Limiting distributions of ESDs of an entirely different type of random matrix---based on uniform Haar measure---have also generated much interest, including the recent work \cite{GKZeitouni2011} proving the Single Ring Theorem (there is an
interesting outlier phenomenon for the Single Ring Theorem as well, see
\cite{BRochet2013}).  It is shown in \cite[Proposition 4]{GKZeitouni2011} that adding polynomially small iid complex Gaussian noise
expands the class of random matrices to which the Single Ring Theorem applies,
essentially removing a hypothesis about the smallest singular value.  This fact
inspired further work \cite{GWZeitouni2011} studying how adding polynomially small iid complex Gaussian noise can change the limiting ESD---in some cases quite dramatically---of a sequence of fixed matrices, and that it turn lead the current paper to study the effects on the ESD of adding small iid \emph{non}-Gaussian random noise.

We will consider the case where $M_n$ is a fixed sequence of $n$ by $n$ complex matrices, to which we will add small random noise to get $A_n = M_n+f(n)X_n$, where $X_n$ is a random complex $n$ by $n$ matrix with iid mean 0, variance $1/n$ entries, and $f(n) \to 0 $ as $n\to \infty$.  We will refer to the case where $f(n)= n^{-\gamma}$ for some $\gamma>0$ as \emph{polynomially small random noise}, and we will refer to the general case as \emph{random noise scaled by $f(n)$}.  The initial motivation for this paper is the fact that, in some natural cases, the ESDs of the perturbed matrix $A_n$ and fixed matrix $M_n$ are very different, even when the random perturbation is very small, e.g. $f(n)=n^{-100}$.  For example, the $n$ by $n$ nilpotent matrix
 \begin{equation}
   \label{eq-tN}
T_n=\left(\begin{array}{lllll}
  0&1&0&\cdots&0\\
  0&0&1&\ddots&\vdots\\
\vdots&\ddots&\ddots&\ddots&0\\
0&\ddots&\ddots&0&1\\
0&\ddots&\ddots&\ddots&0
\end{array}
\right)\,.\end{equation}
has only zero as an eigenvalue, with multiplicity $n$.  However, if one sets $f(n)= n^{-\gamma}$ for some $\gamma > 0$ and $X_n$ has iid complex Gaussian entries scaled by $1/n$, then the ESD of $A_n= M_n+ f(n) X_n$ converges in probability to the uniform distribution on the unit circle $\{z: \abs{z} = 1\}$ (proven in \cite{GWZeitouni2011}; see also \cite{Sniady2002}).  Figure~\ref{Tnfig} plots  the eigenvalues for $\gamma =10$ and $n=$ 50, 500, and 5000.

\begin{figure}
{\bf \large \qquad $n=50$ \hfill $n=500$ \hfill $n=5000$\qquad\quad\color{white}.\color{black}}
\parbox{7in}{\hspace{-1.5cm} \scalebox{.4}{\includegraphics{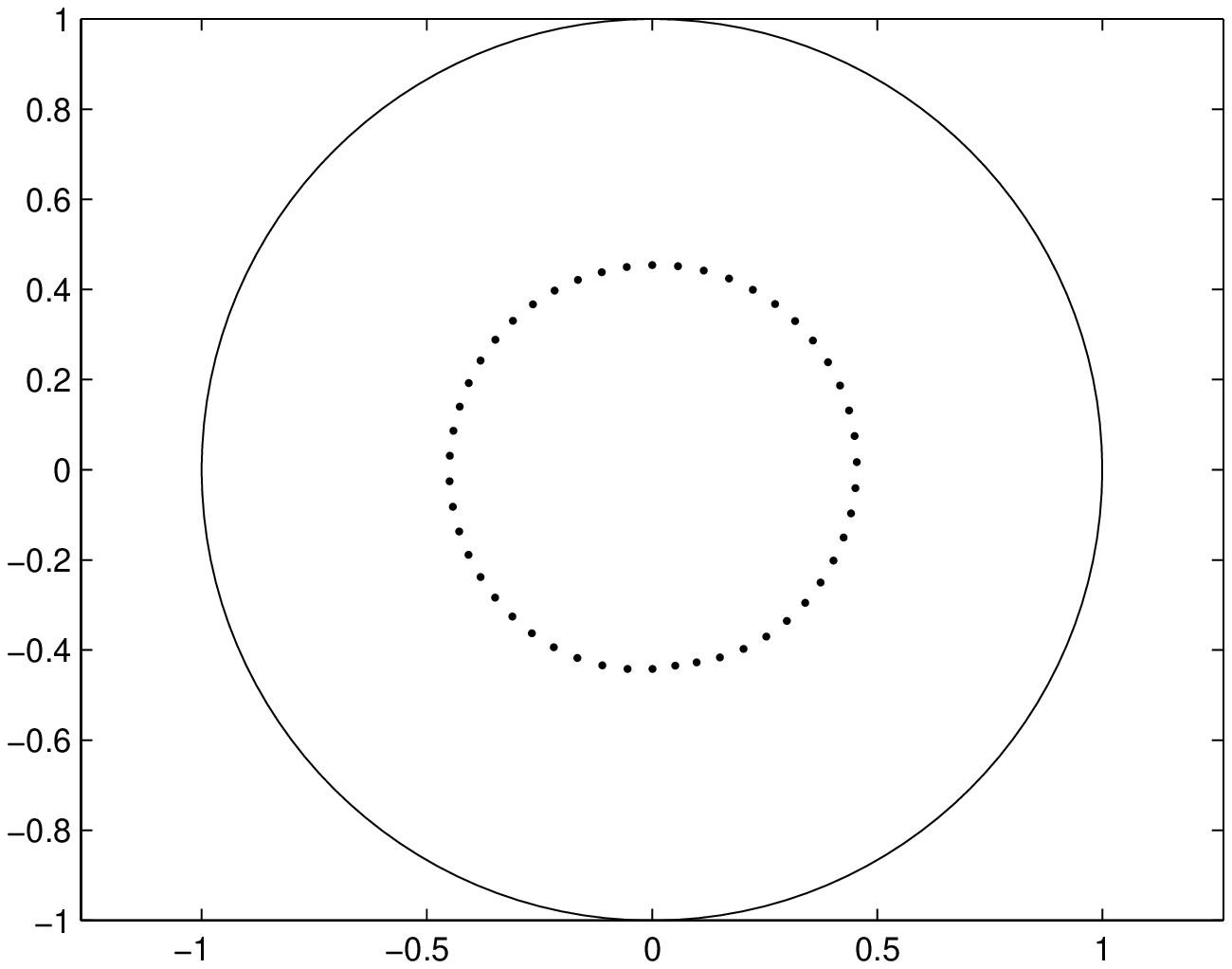}}\scalebox{.4}{\includegraphics{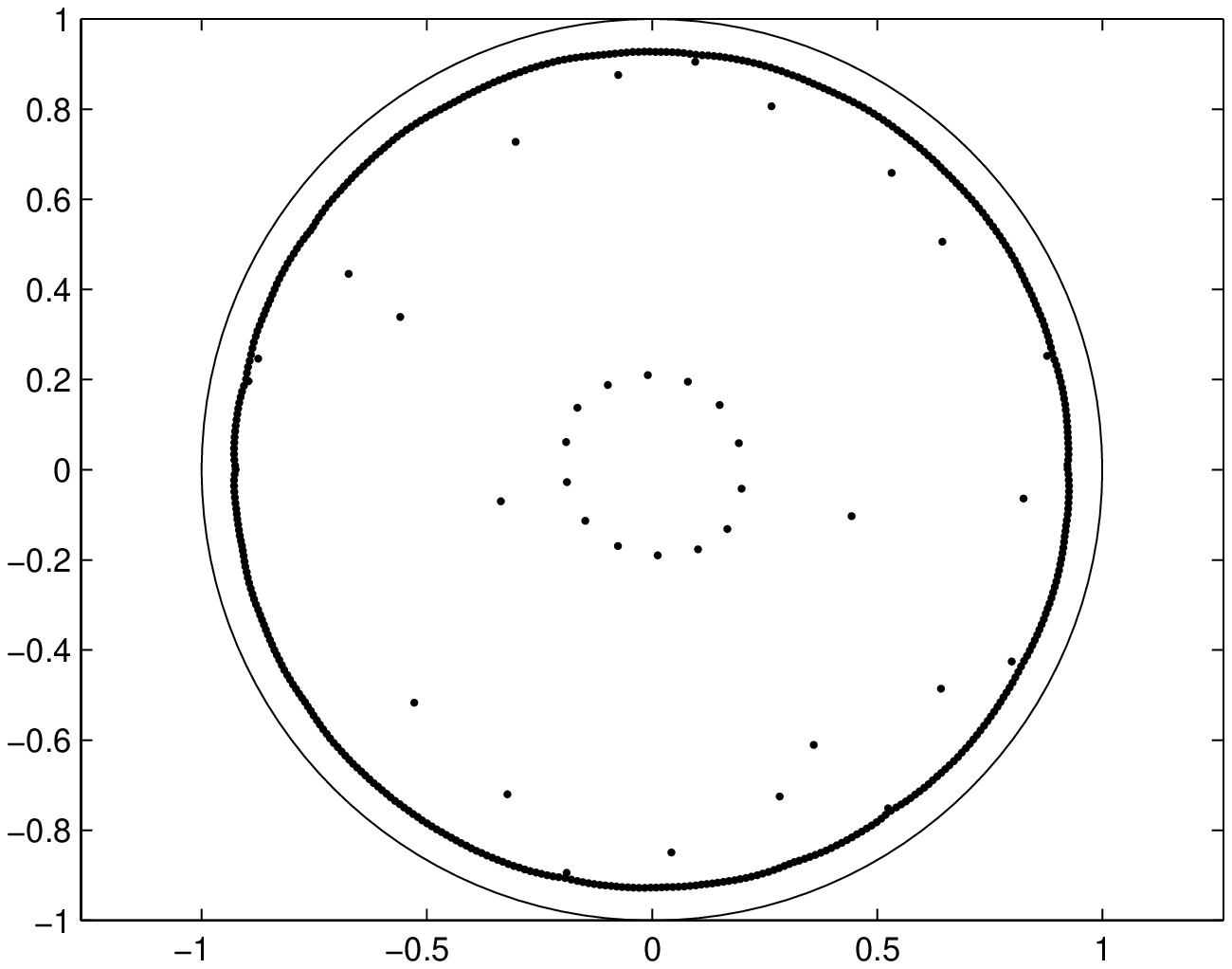}}\scalebox{.4}{\includegraphics{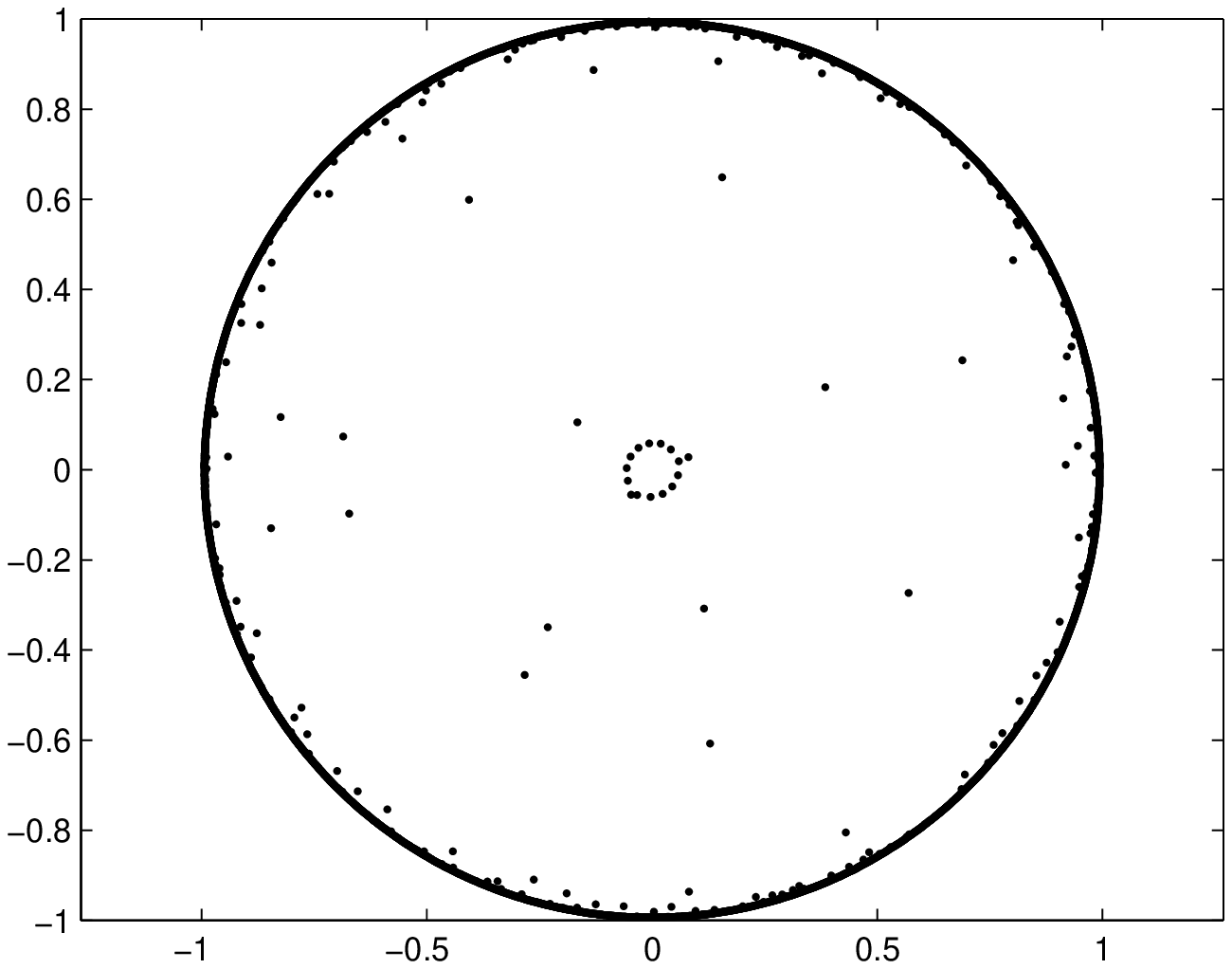}}}
\caption{The eigenvalues of $T_n$ (see equation~\eqref{eq-tN}) plus $n^{-10}X_n$, where $X_n$ is an $n$ by $n$ random iid Gaussian matrix.  A circle is of radius 1 centered at the origin is plotted for comparison (not visible when $n=5000$). With no perturbation, all eigenvalues of $T_n$ equal zero.}
\label{Tnfig}
\end{figure}

Interestingly, the ESD of $T_n$ remains unstable even after polynomially small noise is added, in the sense that a low-rank perturbation (namely rank $o(n)$) of $T_n + n^{-\gamma} X_n$ can change the limiting ESD---see \cite[Corollary~8]{GWZeitouni2011}.  This contrasts with low-rank perturbations of the Circular Law, in which case any rank $o(n)$ perturbation added to the random matrix $X_n$ still has ESD that converges to uniform on the unit disk (see \cite[Corollary~1.12]{TVuK2008}, \cite{Chafai2010}, and \cite{Bordenave2011}).

The matrix $T_n$ shows that the ESD can be very sensitive to small perturbations (this has been noted before; see \cite{Sniady2002,GWZeitouni2011}).  In this paper, we approach the related universality question: ``Is the ESD of a fixed matrix sensitive to the \emph{type} of randomness in a small perturbation?''  For example, in Figure 1, would the ESD plots look the same if the perturbation $X_n$ had iid entries that were Bernoulli $+1/\sqrt n$ or $-1/\sqrt n$ each with probability $1/2$, rather than complex Gaussian?

The first step towards answering this question was taken in \cite[Remark 3]{GWZeitouni2011}, where it is noted (thanks to a comment by R. Vershynin) that in fact the main result of \cite{GWZeitouni2011} extends to the case where the noise matrix has entries that are iid and possess a bounded density.  
Of course, the bounded density assumption excludes Bernoulli random matrices.  The approach in the current paper will not require entries to have bounded density.

In \cite{TVuK2008}, Tao and Vu (with an appendix by Krishnapur) develop a general replacement principle (Theorem~\ref{TVuK2.1} below) that shows convergence of ESDs for random matrix models $A_n$ and $B_n$ if the log-determinants of $A_n+zI_n$ and $B_n+zI_n$ converge for almost every fixed complex number $z$, where $I_n$ is the $n$ by $n$ identity matrix.  This is the framework for the approach in the current paper: if a small perturbation does not change the log-determinant of $A_n+zI_n$, we can use the replacement principle to prove convergence of the ESDs.  

Our focus is on matrices $M_n+zI_n$ where some (usually many) of the rows satisfy the following stability condition for almost every complex number $z$.

\begin{definition}[$\epsilon$-stable]
 \label{stabdef}
A set of vectors $\{\vv 1, \vv 2,\dots, \vv k\}$ is $\epsilon$-\emph{stable} if
$$\dist(\vv j, \Span\{\vv i: 1 \le i \le k, i \ne j\}) \ge \epsilon$$
for all $1\le j\le k$.  In general, $\epsilon$ will be a function of $n$ and other parameters. 
\end{definition}

The $\epsilon$-stable property is reasonably general; for example, a random matrix $R_n$ with iid mean zero, variance $1/n$ entries---thus the row vectors each have expected norm one---contain a set of $n(1-o(1))$ rows that are $\epsilon$-stable for $\epsilon \ge n^{-1/12}/2$ (see Proposition~\ref{p:manystable}); this is also true if $zI_n$ is added to the matrix.  

The $\epsilon$-stability property quantifies the smallest amount one vector would have to be perturbed in order to fall into the span of the remaining vectors.  Intuitively, one might think that the ESD of a matrix with all rows $\epsilon$-stable would not change much under a perturbation that was much smaller than $\epsilon$.  For example, the set of all rows of a diagonal matrix $D_n$ plus $zI_n$ is always at least $\Theta(1)$-stable for almost every $z\in \bb C$ (note $z$ is a constant); thus, one would expect (correctly) that a small $o(1)$ perturbation of $D_n$ has no effect on the limiting ESD (this follows from the Ger{\v s}gorin Circle Theorem \cite{Gersgorin1931,Varga2004}, for example).

However, having many $\epsilon$-stable rows is not the whole story.  By inductive computation, the first $n-1$ rows of the matrix $T_n+zI_n$ are $\Theta(1)$-stable (see Lemma~\ref{Tbn-epstab}); and yet, a small perturbation results in a dramatic change to the ESD as shown in Figure~\ref{Tnfig}. The issue is that when $\abs z$ is small, the last row of $T_n+zI_n$ is only distance $O(|z|^n)$ from the span of the first $n-1$ rows, allowing a small perturbation to produce large changes in the ESD.  

It turns out that we can use bounds on the smallest singular value from \cite{TVu2008} and the replacement principle approach from \cite{TVuK2008} to ignore a small fraction of the rows (in fact, any number $g(n)= o(n/\log n)$ can be ignored).  This allows us to use the $\epsilon$-stability property on the remaining rows to show that the limiting ESD does not depend on the type of randomness in the perturbation.  Our main result (Theorem~\ref{mainthm}) shows for a large class of matrices $M_n$ and for small perturbations that while the ESDs of $M_n+f(n)X_n$ and $M_n$ may differ, the limiting distribution of the ESD of $M_n+f(n)X_n$ is unchanged if the random noise $X_n$ is replaced by a different random matrix ensemble with iid mean 0, variance 1 entries.

\begin{theorem}[Universality of small random noise]\label{mainthm}
Let $M_n$ be sequence of complex $n$ by $n$ matrices satisfying  
\begin{equation}
 \label{e:Mcond}
\sup_n \frac 1n \fnorm{M_n}^2 <\infty.
\end{equation}
Let $x$ and $y$ be complex random variables with mean 0 and variance 1, and let $\Phi_n$ and $\Psi_n$ be $n$ by $n$ matrices having iid entries $x/\sqrt n$ and $y/\sqrt n$, respectively.

Let $A_n=M_n+n^{-\gamma}\Phi_n$ and let $B_n=M_n + n^{-\gamma}\Psi_n$, where $\gamma>1.5$ is a constant.
Assume for almost every $z \in \bb C$ that there is a set
$S$ of at least $n-n/\log^{1.1} n$ rows of $M_n+zI$ that is $\epsilon$-stable, where 
$$\frac{n^{3/4-\gamma/2}\log n}{\epsilon} \to 0\mbox{ as } n\to\infty.$$
Then $\mu_{A_n}- \mu_{B_n}$ converges in probability to zero as $n\to\infty$. 
\end{theorem}
The function $n/\log^{1.1} n$ above 
can be replaced by any function that is $o(n/\log n)$ without changing the proof.  Also, 
while the constraint $\gamma>1.5$ is needed here, it is likely an artifact.

In Section~\ref{S:Tnapp} we will use Theorem~\ref{mainthm} to compute the exact limiting ESD for two families of fixed matrices $M_n$; both results are new. The two families are block-diagonal matrices $M_n$, where the diagonal blocks each equal $T_k$ for some value $k$.  In Theorem~\ref{smsmallblocks}, when the diagonal blocks are small ($k\ll \log n$), all the rows in the matrix $M_n$ are $\epsilon$-stable with $\epsilon$ large enough that the limiting ESD is equal to the limiting ESD of the original matrix $M_n$ (namely all zeros).  In Theorem~\ref{smallblocks}, when the diagonal blocks are large ($k\gg \log n$), Theorem~\ref{mainthm} shows that the limiting ESD $M_n$ plus any polynomial small random noise is equal to the limiting ESD of $T_n$ plus complex Gaussian polynomially small random noise (which is uniform on the unit circle $\{z: \abs z =1\}$ by \cite{GWZeitouni2011}, see also \cite{Sniady2002}).    These families of block-diagonal matrices were introduced in \cite{GWZeitouni2011}, where the case 
$k=c \log n$ for $c$ a positive constant was also studied.  In \cite{GWZeitouni2011}, it was shown that the limiting spectral radius when $k=c\log n$ of the block diagonal matrix $M_n$ plus random noise scaled by $n^{-\gamma}$ with $\gamma > 5/2$ is strictly less than 1, with probability approaching 1 as $n\to\infty$.  The same families of fixed matrices are also being studied in work-in-progress by Feldheim, Paquette, and Zeitouni \cite{FPZeitouni2014} using a very different approach than that used in the current paper.  Feldheim, Paquette, and Zeitouni \cite{FPZeitouni2014} expect to prove theorems similar to Theorem~\ref{smsmallblocks} and Theorem~\ref{smallblocks} with their methods, and they are optimistic that their methods will also lead to an exact computation of the currently unknown limiting distribution of the ESD in the case where $k=c \log n$ for constant $c$.  (Note that the methods of the current paper do not directly apply when $k=c\log n$, since then there too many rows (namely $n/(c\log n)$)
 that must be excluded from the $\epsilon$-stable set in order for $\epsilon$ to be large enough.)

In \cite{Sniady2002}, it is shown that there exists a scaling $t_n$ of iid complex Gaussian noise (with $t_n \to 0$) such that ESD of the matrix $M_n$ plus the $t_n$-scaled Gaussian noise converges almost surely to the Brown measure.  No bounds on $t_n$ are given, however.  In \cite{GWZeitouni2011} it is shown that polynomially small noise is a sufficient: the distribution of the ESD of a matrix $M_n$ plus polynomially small iid complex Gaussian noise converges in probability to the Brown measure of the matrix $M_n$, so long as the matrix $M_n$ and the Brown measure each satisfy a certain regularity property.  
Theorem~\ref{mainthm} shows that, if $M_n$ satisfies \eqref{e:Mcond} and the $\epsilon$-stability condition and random noise is polynomially small, then the requirement that the perturbation $\Phi_n$ be complex Gaussian may be removed: in fact, any $\Phi_n$ with iid mean zero, variance one entries will suffice.  Theorem~\ref{mainthm} has an additional benefit in that it applies to cases where the ESD does \emph{not} converge to the Brown measure; in fact Theorem~\ref{smsmallblocks} is an example of just such a situation.  The proof of Theorem~\ref{mainthm} uses an approach that does not use the free probability machinery that features in \cite{Sniady2002,GWZeitouni2011}.

Topics are organized as follows.  We can apply Theorem~\ref{mainthm} to our
motivating question, proving that the limiting ESD of $T_n$ plus polynomially
small random noise is always uniform on the unit circle; see
Section~\ref{S:Tnapp}.  In Section~\ref{S:Tnapp} we also discuss the
block-diagonal class of matrices generalizing $T_n$ (introduced in \cite{GWZeitouni2011}) and use Theorem~\ref{mainthm} to compute the limiting ESD in two cases where it was previously unknown. In Section~\ref{S:TV}, we will discuss the replacement principle approach to proving universality developed by Tao and Vu
\cite{TVuK2008} (see also \cite{PjmwSparseRandMat2011}).  In
Section~\ref{S:stab}, we will prove that small perturbations of
$\epsilon$-stable sets of vectors remain $(\epsilon/2)$-stable, and we will show how this relates to the tools in Section~\ref{S:TV}.  The proof of Theorem~\ref{mainthm} is in Section~\ref{ss:mainthmproof}.  

No effort is made to optimize constants, and we will often choose explicit
constants to make computations clearer.  All logarithms are natural unless
otherwise noted.
Also, we will use $\fnorm{A}=\tr(A^*A)^{1/2}$ to denote the Hilbert-Schmidt norm (also called the Frobenius norm).

\section{Application to a class of non-normal matrices} \label{S:Tnapp} \label{S:Tbn}

In this section, we will give sketches of how the main theorem
(Theorem~\ref{mainthm}) can be applied to a class of nilpotent matrices
generalizing $T_n$ that has interesting behaviors when
small random noise is added.  These ESDs of these matrices plus small random noise were studied in \cite{GWZeitouni2011} (see also \cite{Sniady2002}) and are being currently studied in \cite{FPZeitouni2014}.  

Let $b$ be a positive integer, and define $T_{b,n}$ to be an $n$ by $n$ block
diagonal matrix with each $b+1$ by $b+1$ block on the diagonal equal to
$T_{b+1}$ (as defined above in Equation~\eqref{eq-tN}).  If $b+1$ does not
divide $n$ evenly, an additional block equal to $T_{k}$ where $k\le b$  is inserted at
bottom of the diagonal (in particular, $k= n-\floor{\frac{n}{b+1}}(b+1)$, and if $k=0$, then no additional block is needed).
Thus, every entry of $T_{b,n}$ is zero except for entries on the superdiagonal  (the superdiagonal is the list of entries with coordinates $(i,i+1)$ for $1\le  i\le n-1$), and the superdiagonal of $T_{b,n}$ is equal to
$$(\underbrace{1,1,\dots,1}_b,0,\underbrace{1,1,\dots,1}_b,0,\dots,
\underbrace{1,1,\dots,1}_b,\underbrace{0,1,1,\dots,1}_{\leq b}).$$
(Note that $T_{b,n}$ was defined slightly differently in \cite{GWZeitouni2011}, in that the last (possibly non-existent) diagonal block contained all zeros.)
Recall that the spectral radius of a matrix is the maximum absolute value of    the eigenvalues.  
In \cite{GWZeitouni2011}, it was proven that the distribution of the ESD of $T_n$ plus polynomially small Gaussian noise converges in probability to uniform on the unit circle, and it was shown for $\gamma > 5/2$ that the spectral radius of $T_{\log n,n}$, plus random noise scaled by $n^{-\gamma}$ 
is 
strictly less than 1, with probability approaching 1 as $n\to\infty$. 

The matrix $T_{b,n}+zI$ has a large set of rows that are $\epsilon$-stable for constant $\epsilon$ (depending on $z$), namely the set of all rows of the form $(0,\dots,0,z,1,0,\dots,0)$, a set having size at least $n-\floor{n/(b+1)}$ (see Lemma~\ref{Tbn-epstab}).
However, the $\epsilon$-stability of the set of \emph{all} rows of $T_{b,n}+zI$
is much smaller, having size $\Theta(\abs{z}^{b+1})$ for small $z$, which is exponentially small when $b = \Theta(n)$ (see Lemma~\ref{Tbn-epstab-all}). 

We can apply our main theorem to prove the following two results about the limiting ESD of $T_{b,n}$ plus polynomially small random noise, for different sizes of $b$.

\begin{theorem}[Small blocks]
 \label{smsmallblocks}
Let $T_{b,n}$ be as defined above, with $b = o(\log n)$, and let $\Phi_n$ be a random matrix
with iid mean 0 variance $1/n$ entries. 

Then, the distribution of the ESD of $T_{b,n}+ n^{-\gamma} \Phi_n$, where $\gamma > 1.5$, converges in probability to the Dirac measure $\delta_0$, with mass 1 at the origin.
\end{theorem} 

\begin{proof}[Sketch]
Blocks of size $b=o(\log n)$ are small enough that the $\epsilon$-stability of
\emph{all} the rows of the matrix is reasonably high, namely $\epsilon >\Omega(
n^{o(1)})$ (see Lemma~\ref{Tbn-epstab-all}).  We can apply the proof
approach for the main result (Theorem~\ref{mainthm}) to show that the distribution of the ESD of $T_{b,n}+n^{-\gamma} \Phi_n$ converges in probability to the ESD of $T_{b,n}$, which has all eigenvalues equal to zero. The full details appear in Section~\ref{S:AppPfs}. 
\end{proof}

In the case where $b_n \to \infty$ (e.g., $b_n=\log\log n$), Sniady's result \cite{Sniady2002} shows that the distribution of the ESD of the perturbation of $T_{b,n}$ converges almost surely to uniform on the unit circle $\{z: \abs z=1\}$ (matching the Brown measure), if one perturbs with random iid complex Gaussian noise scaled by some particular $t_n$, where $t_n\to 0$.  Theorem~\ref{smsmallblocks} shows that, for polynomially small random noise, the distribution of the ESD of the perturbed matrix does \emph{not} converge to the Brown measure, but rather converges in probability to the limiting ESD of $T_{b,n}$ without perturbation (namely, the Dirac measure $\delta_0$).  Two interesting questions one might ask are what is the scaling $t_n$ so that the ESD converges to uniform on the unit circle, and whether universality holds for that scaling $t_n$. 

\begin{remark}
One could conceive of a of Theorem~\ref{mainthm} where the random noise was scaled by an arbitrary function $f(n)$ with $f(n)\to 0$ as $n\to\infty$, rather than by $n^{-\gamma}$.  The singular value bound (from \cite{TVu2008}, which is restated in Theorem~\ref{lsv0}) hold in a useful form if $f(n) \ge n^{-\gamma}$ (e.g. $f(n) = 1/\log n$), but it would not be useful for exponentially small $f(n)$.  Conversely, the $\epsilon$-stability condition would likely be fine for $f(n)\le n^{-\gamma}$ (including $f(n)$ exponentially small) but is likely to be problematic when $f(n) \gg  n^{-\gamma}$ (e.g. $f(n)=1/\log n$), requiring in the latter case that $\epsilon$ is much larger than practical.  Nonetheless, in principle, we expect---supported by some computer experimentation---that universality of the form in Theorem 2 is very robust and should hold without conditions on the size of the random noise.
\end{remark}

\begin{theorem}[Large blocks]
 \label{smallblocks}
Let $T_{b,n}$ be as defined above with $b\gg \log n$ (i.e., $b/\log n \to
\infty$ as $n \to \infty$), and let $\Phi_n$ be a complex iid random matrix where each entry has mean 0 and variance $1/n$.

Then, the distribution of the ESD of $T_{b,n}+ n^{-\gamma} \Phi_n$, where $\gamma > 1.5$, converges in probability to the uniform measure on the unit circle $\{z: \abs z = 1\}$.  
In particular, by setting $b=n-1$, we have that the distribution of the ESD of $T_n+n^{-\gamma}\Phi_n$, where $\gamma > 1.5$, converges in probability to the uniform measure on the unit circle.
\end{theorem}

\begin{proof}[Sketch]
In this case, there are only $o(n/\log n)$ rows of $T_{b,n}$ that differ from the corresponding rows of $T_n$.  We can ignore these rows using the replacement principle approach from \cite{TVuK2008} (see Section~\ref{S:TV}), thus showing that the distribution of the ESD of $T_{b,n}+n^{-\gamma} \Phi_n$ converges in probability to the distribution of the ESD of $T_n + n^{-\gamma}\Phi_n$.  

Next, we can apply Theorem~\ref{mainthm}, noting that if the last row is excluded, the remaining rows are $\epsilon$-stable for constant $\epsilon$ (Lemma~\ref{Tbn-epstab}), to show that the ESD of $T_{n}+n^{-\gamma} \Phi_n$ converges in probability to the ESD of $T_n + n^{-\gamma}\Psi_n$, where $\Psi_n$ has iid complex Gaussian entries with mean 0 and variance $1/ n$.  This ESD in turn converges in probability to uniform on the unit circle by \cite{GWZeitouni2011}.
\end{proof}

\section{The replacement principle approach to proving universality}\label{S:TV}

In \cite{TVuK2008}, Tao and Vu (with an appendix by Krishnapur) prove a general result giving sufficient conditions for the ESDs of two matrices to become close to each other.

\begin{theorem}\emph{\cite{TVuK2008}} \label{TVuK2.1} Suppose for each $n$
that $A_n, B_n \in \bb M_n(\bb C)$ are ensembles of random matrices.  Assume that
\begin{itemize}
\item[\emph{(i)}]
The expression
\begin{equation*}
\frac1{n} \fnorm{A_n}^2 + \frac1{n}\fnorm{B_n}^2
\end{equation*}
is bounded in probability (resp. almost surely).

\item[\emph{(ii)}]
For almost all complex numbers $z$,
$$\frac1n \log \abs{\det(A_n + zI)} - 
\frac1n \log \abs{\det(B_n + zI)} $$
converges in probability (resp. almost surely) to zero.  In particular, for each fixed $z$, these
determinants are non-zero with probability $1-o(1)$ for all $n$ (resp. almost surely non-zero for all but finitely many $n$).
\end{itemize}
Then, $\mu_{A_n} - \mu_{B_n} $ converges
in probability (resp. almost surely) to zero.
\end{theorem}

Note that Theorem~\ref{TVuK2.1} makes no assumption about the type of randomness in $A_n$ and $B_n$, or even whether the entries are independent.  We will eventually require independence of the entries in order to use bounds on the smallest singular value.  

\begin{lemma}\label{Aasb}
  For $\gamma > 0$ and $\Phi_n$ and $M_n$ as in Theorem~\ref{mainthm}, the matrix $A_n=M_n+ n^{-\gamma}\Phi_n$ satisfies $\frac1n\fnorm{A_n}^2$ 
is almost surely bounded, and the same statement holds with $A_n$ replaced by $B_n=M_n + n^{-\gamma}\Psi_n$.
\end{lemma}

\begin{proof}
 Paraphrasing \cite[Lemma~1.9]{TVuK2008}, the result follows by combining \eqref{e:Mcond} with the triangle inequality, and using the law of large numbers along with the fact the second moments of the entries in $\sqrt n\Phi_n$ are finite.
\end{proof}

As noted in \cite{TVuK2008}, one fact that makes Theorem~\ref{TVuK2.1} particularly useful is that there are a number of different ways to express $\abs{\det(A)}$.  For example, for an $n$ by $n$ matrix $A$
\begin{equation}\label{detforms}
\abs{\det(A)} = \prod_{i=1}^n \abs{\lambda_i} = \prod_{i=1}^n \sigma(A) = \prod_{i=1}^n d_i(A)
\end{equation}
where $\lambda_1,\dots,\lambda_n$ are the eigenvalues of $A$ (with multiplicity), where $\sigma_1(A)\ge \sigma_2(A) \ge\dots\ge\sigma_n(A)\ge 0$ are the singular values of $A$, and where $d_i(A)$ is the distance from the $i$-th row of $A$ to the span of the first $i-1$ rows. Combining Equation~\eqref{detforms} with a result such as Lemma~\ref{Aasb} reduces proving universality of the ESD to a question about the distance from a perturbed vector to a span of perturbed vectors. 
In particular, one can prove Theorem~\ref{mainthm} using Lemma~\ref{Aasb}, Theorem~\ref{TVuK2.1}, and the following proposition:

\begin{proposition}\label{propdist}
Let $X_i$ be the rows of $A_n+zI$ and let $Y_i$ be the rows of $B_n+zI$, where $A_n$ and $B_n$ are as in Theorem~\ref{mainthm}. 
For almost all complex numbers $z$,
\begin{equation}\label{e:Goal}
\frac1n \sum_{i=1}^n \log \dist(X_i,\Span\{X_1,\dots,X_{i-1}\}) - 
 \log \dist(Y_i,\Span\{Y_1,\dots,Y_{i-1}\}) 
\end{equation}
converges in probability
to zero.
\end{proposition}

Proving Proposition~\ref{propdist} is the goal of Subsection~\ref{S:singvals},  Section~\ref{ss:mainthmproof}, and Section~\ref{S:stab}.

\subsection{Singular values for polynomially small random noise}
\label{S:singvals}

As shown in \cite{TVuK2008}, a bound on the singular values allows one to bound the highest-dimensional distances in \eqref{e:Goal}.
 
For an $n$ by $n$ matrix $M_n$, let $\|M_n\|$ denote the spectral norm of $M_n$ (which is also the largest singular value of $M_n$), namely
$$\| M_n\| = \sup_{|\vv{}|=1} |M_n\vv{}|. $$
For a matrix $M$, let the singular values be denoted $$\sigma_1(M) \ge \sigma_2(M)\ge\dots\ge \sigma_n(M) \ge 0.$$

\begin{theorem}[Least singular value bound]\label{lsv0} \emph{\cite{TVu2008}}  Let $A, B, \gamma$ be positive constants,
and let $x$ be a complex-valued random variable with non-zero
finite variance (in particular, the second moment is finite). Then
there are positive constants $C_{1}$ and $C_{2}$  such that the
following holds: if $X_{n}$ is the random $n$ by $n$ matrix whose
entries are iid copies of $x$, and $M_n$ is a deterministic $n$ by $n$ matrix with spectral norm $\norm{M_n} \le n^{B-\gamma}$, then,
$$ \P( \sigma_n(M_n + n^{-\gamma}X_n) \le n^{-C_{1}-\gamma} ) \le C_{2} n^{-A}.$$
\end{theorem}

The above is a restatement of \cite[Theorem~2.1]{TVu2008} using the fact that $\sigma_n(M_n+n^{-\gamma}X_n) \le n^{-C_1-\gamma}$ is equivalent to 
$\sigma_n(n^{\gamma}M_n+X_n) \le n^{-C_1}$.  Applying Theorem~\ref{lsv0} to the matrices $A_n$ and $B_n$ from the statement of Theorem~\ref{mainthm}, we have with probability 1 that
\begin{equation}\label{svlower}
 \sigma_n(A_n), \sigma_n(B_n) \ge n^{-O(1)},
\end{equation}
for all but finitely many $n$. As pointed out in \cite{TVuK2008}, a polynomial upper bound 
\begin{equation}\label{svupper}
 \sigma_n(A_n), \sigma_n(B_n) \le n^{O(1)} 
\end{equation}
holds with probability 1 for all but finitely many $n$. (The upper bound follows from \eqref{e:Mcond}, the bounded second moments of $x$ and $y$, and the Borel-Cantelli lemma.)

\begin{lemma}\label{lem:Hdim}
Let $X_i$ be the rows of $A_n+zI$ and let $Y_i$ be the rows of $B_n+zI$, where $A_n$ and $B_n$ are as in Theorem~\ref{mainthm}.
For almost all complex numbers $z$ and with probability 1,
$$\abs{\log \dist(X_i,\Span\{X_1,\dots,X_{i-1}\})} \le O(\log n)$$
and
$$\abs{\log \dist(Y_i,\Span\{Y_1,\dots,Y_{i-1}\})}
\le O(\log n)$$
for all but finitely many $n$.
\end{lemma}

The above lemma follows from \eqref{svlower}, \eqref{svupper}, and 
\cite[Lemma~A.4]{TVuK2008}.  When bounding a sum such as \eqref{e:Goal}, Lemma~\ref{lem:Hdim} allows one to ignore up to $o(n/\log n)$ rows, since a sum of $o(n/\log n)$ numbers that are each at most $O(\log n)$ converges to zero.

\section{Proof of Proposition~\ref{propdist}}
\label{ss:mainthmproof}

We will now prove Proposition~\ref{propdist}, thereby completing the proof of Theorem~\ref{mainthm}.  In the proof, we will use tools described above along with Proposition~\ref{corpolymatA} below which we will prove in Section~\ref{S:stab} as Proposition~\ref{corpolymat}.

Note that Proposition~\ref{propdist} can be re-stated in terms of determinants (see \eqref{detforms}), and thus Proposition~\ref{propdist} is equivalent to the same statement with the rows and corresponding columns of the matrices re-ordered in the same way (since re-ordering rows and columns has no effect on the determinant).  
Recall that $A_n+zI=M_n+n^{-\gamma}\Phi_n+ zI$ and
$B_n+zI=M_n+n^{-\gamma}\Psi_n+ zI$, and note that re-ordering rows and columns of $\Phi_n$ and $\Psi_n$ has no effect since the entries are iid.  Thus, in proving Proposition~\ref{propdist}, we may re-order the rows and corresponding columns as is convenient.  In particular, we may re-order so that $Z_1,\dots,Z_m$ are the first $m$ rows of $M_n+zI$ and are $\epsilon$-stable with the same $\epsilon$ from the assumption in Theorem~\ref{mainthm}.  We may further require that the re-ordering satisfies $\vnorm{Z_1} \ge \vnorm{Z_2} \ge \dots \ge \vnorm{Z_m}$.

\newcommand{\Corpolymatstmt}[1][]{
Let $M_n#1$ be an $n$ by $n$ matrix, let $\Phi_n$ and $\Psi_n$ be $n$ by $n$ complex matrices with each row a random complex vector with mean 0 and variance 1, and let $Z_1,\dots,Z_m$ be the first $m$ rows of $M_n#1$, where $m=\floor{n-\frac{2n}{\log^{1.1}n}}$.  Assume that $\vnorm{Z_1} \ge \vnorm{Z_2} \ge \dots \ge \vnorm{Z_m}$.
If $\{Z_1,\dots, Z_m\}$ is $\epsilon$-stable and  there exists a constant $\gamma > 1.5$ so that
$$\delta_{n,\epsilon}:= \frac{n^{-\gamma/2+3/4}\log^{1/4}(n) \sqrt{\max_{1\le i\le m}
\{1,\vnorm{Z_i}\}}}{\epsilon} \to 0 \qquad \mbox{ as } n \to \infty,$$

then with probability at least $1-1/\log n$ we have, 
$$\abs{\frac1n \sum_{i=1}^{m} \log d_i(M_n#1+n^{-\gamma}\Phi) - \log
d_i(M_n#1+n^{-\gamma}\Psi)} \le
20\deltane
$$
for all sufficiently large $n$, where $d_i(A)$ is the distance from the $i$-th
row of a matrix $A$ to the span of rows $1,2,\dots,i-1$.
}

\begin{proposition} \label{corpolymatA}
\Corpolymatstmt[+zI]
\end{proposition}

The proof approach Proposition~\ref{corpolymatA} is to add up the errors from perturbing each row with polynomially small random noise. If the set of rows is $\epsilon$-stable with $\epsilon$ enough larger than the polynomially small random noise, then the sum of all the errors can be shown to be small.  See Section~\ref{S:stab} and Proposition~\ref{corpolymat} (which is a restatement of Proposition~\ref{corpolymatA}) for details.

To prove Proposition~\ref{propdist}, we will first exclude rows from $\{Z_1,\dots, Z_m\}$ that are large.  From the assumption \eqref{e:Mcond} that $\sup_n \frac 1n \fnorm{M_n}^2 < \infty$, we know that all but at most $\frac{n}{\log^{1.1}n}$ rows of $M_n+zI$ satisfy $\vnorm{Z_i} > O(\log^{.55} n)$.  We will exclude the first $\frac{n}{\log^{1.1}n}$ rows from the stable set, focusing instead on the set $\{Z_i: \frac{n}{\log^{1.1}n} < i \le m\}$ which is $\epsilon$-stable and satisfies
$$\vnorm{Z_i} \le O(\log^{.55} n),\mbox{ for } \frac{n}{\log^{1.1} n \le i \le m}.$$
Next, we will re-order the rows and corresponding columns so that the set  $\{Z_i: \frac{n}{\log^{1.1}n} < i \le m\}$ is the first $m'=\floor{m - \frac{n}{\log^{1.1}n}} = \floor{n - \frac{2n}{\log^{1.1}n}}$ rows of the re-ordered matrix.

Recall that to prove Proposition~\ref{propdist} we must show for almost every $z\in \bb C$ that
\begin{equation}\label{sum2bound}
\frac1n \sum_{i=1}^{n} \log d_i(M_n+zI+n^{-\gamma}\Phi) - \log d_i(M_n+zI+n^{-\gamma}\Psi)
\end{equation}
converges to zero in probability, where $d_i(A)$ is the distance from the
$i$-th row of matrix $A$ to the span of the first $i-1$ rows.

Lemma~\ref{lem:Hdim} shows that the portion of the sum in \eqref{sum2bound} for
$m' < i \le n$
converges to zero in probability, and Proposition~\ref{corpolymatA}  proves that
portion of the sum in \eqref{sum2bound} for $1 \le i \le
m'=\floor{n-2n/\log^{1.1} n}$ also converges to zero in
probability (note that $\deltane < \frac{n^{3/4-\gamma/2}\log
n}{\epsilon}\to 0$ as $n\to\infty$ by assumption). \hfill{$\square$}

\section{The stability approach for small random noise}\label{S:stab}

Using Theorem~\ref{TVuK2.1}, we see that one way to prove universality is to control
quantities of the form
$$\dist\lt(Z_i+f(n)\phi_n, \Span\{Z_1+f(n)\phi_1,\dots,Z_{i-i}+f(n)\phi_{i-1}\}\rt),$$
where $Z_j$ is a row of $M_n$ and $\phi_j$ is a random $n$-dimensional vector with iid mean zero, variance $1/n$ entries.

As a warm-up, we show in the proposition below that most matrices satisfy the $\epsilon$-stable condition given in Theorem~\ref{mainthm} when one takes $\gamma > 5/3$.  

\begin{proposition} \label{p:manystable}
 Let $R_n$ be a random matrix where the entries are iid copies of $x/\sqrt n$, where $x$ is a mean zero, variance 1 complex random variable.  Then, with probability one, $R_n$ contains a set of $n-n^{5/6}$ rows that is $(n^{-1/12}/2)$-stable, for all but finitely many $n$.  Furthermore, the same result holds for $R_n+zI_n$ where $z$ is a fixed complex number and $I_n$ is the $n$ by $n$ identity matrix.
\end{proposition}

\begin{proof}
Since we may take $z=0$, it suffices to prove the result for $R_n+zI_n$.
We will show that the first $m$ rows of $R_n+zI_n$ form a stable set.  Let
$d_{i,m}(R_n +zI_n)$ be the distance from the $i$-th row to span of the first $m$
rows not including row $i$.  Assuming that $m \le n-n^{5/6}$, we can apply
\cite[Proposition~4.2]{PjmwSparseRandMat2011} (see also
\cite[Proposition~5.1]{TVuK2008}) to get 
$$\Pr{d_{i,m}(R_n+zI_n) \le n^{-1/12}/2} \le6\exp(-n^{1/2}),$$
for each $1\le i \le m$, for all sufficiently large $n$.  By the union bound, the probability that any of the $m=n-n^{5/6}$ rows in the stable set satisfy $d_{i,m}(R_n+zI_n) \le n^{-1/12}/2$ is at most $6m\exp(-n^{1/2})\le 6n\exp(-n^{1/2})$.  This probability is summable in $n$, and so the Borel-Cantelli lemma completes the proof.
\end{proof}

\subsection{Small perturbations of one row}

\begin{lemma}\label{stablem1}
Let $Z_1,\dots,Z_i$ and $\phi_i$ be $n$-dimensional complex vectors, and let
$f_i(n)$ be a non-negative real function of $n$.  Then
$$\abs{
\dist( Z_i, \Span\{Z_1,\dots,Z_{i-1}\})-
\dist( Z_i+f_i(n)\phi_i, \Span\{Z_1,\dots,Z_{i-1}\})
}\le f_i(n) \vnorm{\phi_i}$$
\end{lemma}

\begin{proof}
 The result follows from the triangle inequality.
\end{proof}

\begin{lemma}\label{stablem2}
Let $Z_1,\dots,Z_i$ and $\phi_1$ be $n$-dimensional complex vectors, and let
$f_1(n)$ be a non-negative real function of $n$.  
Assume that 
$$ d_1:=\dist(Z_1, \Span\{Z_2,\dots,Z_{i-1}\}) > f_1(n)\vnorm{\phi_1}.$$

Then
\begin{align*}
&\hspace{-1.5cm}\abs{
\dist( Z_i, \Span\{Z_1,\dots,Z_{i-1}\})-
\dist( Z_i, \Span\{Z_1+f_1(n)\phi_1,Z_2,\dots,Z_{i-1}\})
}\\
&\le \vnorm{Z_i} f_1(n) \vnorm{\phi_1}\lt( \frac{4d_1 + 2f_1(n)\vnorm{\phi_1}
}{(d_1 - f_1(n)\vnorm{\phi_1})^2} \rt).
\end{align*}
Furthermore, if one assumes that $d_1 \ge \epsilon > 2 f_1(n)\vnorm{\phi_1}$,
then one can simplify the upper bound noting that
$$
\lt( \frac{4d_1 + 2f_1(n)\vnorm{\phi_1}
}{(d_1 - f_1(n)\vnorm{\phi_1})^2} \rt) 
\le 
\frac{20}{\epsilon}
.$$
\end{lemma}

\begin{proof}
\newcommand\utilde{\widetilde u_1}
 Let $u_2,\dots,u_{i-1}$ be an orthonormal basis for $Z_2,\dots,Z_{i-1}$, let
$$u_1 = Z_1 - \sum_{j=2}^{i-1} u_j \angles{Z_1,u_j},$$
where $\angles{\cdot,\cdot}$ is the standard complex inner product (so $\vnorm{u_1}=d_1$).  Also, let
$$\utilde = Z_1 + f_1(n)\phi_1 - \sum_{j=2}^{i-1}u_j \angles{Z_1+f_1(n)\phi_1, u_j}.$$

Note that
\begin{align*}
&\hspace{-.5cm}\abs{
\dist( Z_i, \Span\{Z_1,\dots,Z_{i-1}\})-
\dist( Z_i, \Span\{Z_1+f_1(n)\phi_1,Z_2,\dots,Z_{i-1}\})
} \\
&\leq \vnorm{ \frac{u_1}{\vnorm{u_1}^2}\angles{Z_i,u_1} - 
\frac{\utilde}{\vnorm{\utilde}^2}\angles{Z_i,\utilde}}, 
\end{align*}
and thus the current lemma can be proven by studying $u_1$, $\utilde$, and
$Z_i$.  Let $e=\utilde-u_1= f_1(n)(\phi_1-\sum_{j=2}^{i-1} u_j
\angles{\phi_1,u_j})$, and note that $\vnorm{e} \le f_1(n)\vnorm{\phi_1}$.
 
We may write
\begin{align*}
&\hspace{-.5cm}\frac{\utilde}{\vnorm{\utilde}^2}\angles{Z_i,\utilde} - \frac{u_1}{\vnorm{u_1}^2}\angles{Z_i,u_1} \\
&= u_1 \frac{\angles{Z_i, u_1}}{\vnorm{u_1+e}^2}+
e \frac{\angles{Z_i, u_1+e}}{\vnorm{u_1+e}^2}+
u_1 \frac{\angles{Z_i, e}}{\vnorm{u_1+e}^2}
- \frac{u_1}{\vnorm{u_1}^2}\angles{Z_i,u_1}\\
&=\frac{u_1}{\vnorm{u_1}^2}\angles{Z_i,u_1}\lt(\frac{\vnorm{u_1}^2}{\vnorm{u_1+e}^2} -1\rt)
+\frac{e\angles{Z_i, u_1+e}+ u_1 \angles{Z_i, e}}{\vnorm{u_1+e}^2}.
\end{align*}

Using Cauchy-Schwartz, the triangle inequality, and the fact that $\vnorm{u_1+e} = \vnorm{u_1}+c_0\vnorm{e}$ for some constant $-1\le c_0 \le 1$, we can compute that the above vector has length at most
\begin{align*}
&\hspace{-.5cm} \vnorm{Z_i} \left(\abs{\frac{\vnorm{u_1}^2 -\vnorm{u_1+e}^2}{\vnorm{u_1+e}^2}}
+ \frac{\vnorm{e}}{\vnorm{u_1+e}} 
+ \frac{\vnorm{u_1}\vnorm{e}}{\vnorm{u_1+e}^2} 
\right)\\
& =\frac{\vnorm{Z_i}}{(\vnorm{u_1}+c_0\vnorm{e})^2} \left(\abs{-2c_0\vnorm{u_1}\vnorm{e} -c_0^2\vnorm{e}^2}
+ \vnorm{e}(\vnorm{u_1}+c_0\vnorm{e}) 
+ \vnorm{u_1}\vnorm{e}
\right)\\
& \le\frac{\vnorm{Z_i}}{(\vnorm{u_1}+c_0\vnorm{e})^2} 
\left((2\abs{c_0}+2)\vnorm{u_1}\vnorm{e} +
\vnorm{e}^2(c_0^2+c_0) 
\right).\\
\end{align*}

Using the the assumption that $\vnorm{u_1}= d_1 > f_1(n)\vnorm{\phi_1}$ and the facts that $-1\le c_0\le 1$ and $\vnorm e \le f_1(n)\vnorm{\phi_1}$, the above distance is at most
\begin{align*}
& \hspace{-.5cm} \vnorm{Z_i}
\frac{\left(4\vnorm{u_1}f_1(n)\vnorm{\phi_1} + 2 f_1(n)^2\vnorm{\phi_1}^2 
\right)}{(\vnorm{u_1}-f_1(n)\vnorm{\phi_1})^2} 
,\\ 
\end{align*}
which completes the proof.
\end{proof}

\subsection{Stability and its changes in the presence of small random noise}

\begin{lemma}
 \label{indlemma}
Let $Z_1,\dots,Z_k$ and $\phi_1$ be $n$-dimensional complex vectors, and let
$f_1(n)$ be a non-negative real function of $n$.  Assume that $\{Z_1,\dots,Z_k\}$
is $\epsilon$-stable and that $\epsilon > 2 f_1(n)\vnorm{\phi_1}$.

Then $\{Z_1+f_1(n)\phi_1,Z_2,\dots,Z_k\}$ is
$$\lt(\epsilon-f_1(n)\vnorm{\phi_1}\max_{2\le i\le k}\lt\{1,\vnorm{Z_i}\lt( \frac{20}\epsilon \rt)\rt\} \rt)\mbox{-stable}.$$
\end{lemma}

\begin{proof}
 By Lemma~\ref{stablem1}, we know that
$$\abs{
\dist( Z_1, \Span\{Z_2,\dots,Z_{k}\})-
\dist( Z_1+f_1(n)\phi_1, \Span\{Z_2,\dots,Z_{k}\})
}\le f_1(n) \vnorm{\phi_1}.$$

For $2\le i \le k$, let $Z_2,\dots,\widehat Z_i,\dots, Z_k$ denote the set $\{Z_j: 2\le j\le k, j \ne i\}$.  By Lemma~\ref{stablem2} we know that
\begin{align}
&\nonumber\hspace{-.5cm}\abs{
\dist( Z_i, \Span\{Z_1,Z_2,\dots,\widehat Z_i,\dots,Z_{k}\})-
\dist( Z_i, \Span\{Z_1+f_1(n)\phi_1,Z_2,\dots, \widehat Z_i, \dots,Z_{k}\})
}\\
&\le \vnorm{Z_i} f_1(n) \vnorm{\phi_1}\lt( \frac{4d_1 + 2f_1(n)\vnorm{\phi_1}
}{(d_1 - f_1(n)\vnorm{\phi_1})^2} \rt),\label{eq:bdinspan}
\end{align}
where $d_1 = \dist(Z_1,\Span\{Z_2,\dots,\widehat Z_i,\dots,Z_k\})$.  By the
$\epsilon$-stable assumption, $d_1 \ge \epsilon > 2 f_1(n)\vnorm{\phi_1}$, so the
bound in \eqref{eq:bdinspan} is at most $ \vnorm{Z_i} f_1(n) \vnorm{\phi_1}(20/\epsilon)$.
\end{proof}

\newcommand\ztilde[1]{\widetilde Z_{#1}}

\begin{lemma}[Continued stability]
 \label{contstab}
Let $Z_1,\dots,Z_k$ and $\phi_1,\dots,\phi_k$ be $n$-dimensional complex vectors,
and let $f_1(n),\dots,f_k(n)$ be non-negative real functions of $n$.  Assume that $\{Z_1,\dots,Z_k\}$ is $\epsilon$-stable and 
that $$20\ge \epsilon > \sqrt{40 k \max_{1\le i\le
k}\{f_i(n)\vnorm{\phi_i}\}\lt(\max_{1\le i\le k}\{1,\vnorm{Z_i}\}+\max_{1\le
i\le k} \{f_i(n)\vnorm{\phi_i}\} \rt)}.$$
Then, for each $1\le j \le k$, we have that
$$\lt\{\ztilde 1, \ztilde 2, \dots, \ztilde j, Z_{j+1}, Z_{j+2},\dots Z_k\rt\}$$
is $(\epsilon/2)$-stable, where $\ztilde i = Z_i + f_1(n)\phi_i$.
\end{lemma}

\begin{proof}
 We will prove the following stronger statement by induction on $j$:
$$\lt\{\ztilde 1, \ztilde 2, \dots, \ztilde j, Z_{j+1}, Z_{j+2},\dots Z_k\rt\}\qquad
\mbox{ is }\lt(\epsilon - \frac{j\epsilon}{2k}\rt)\mbox{-stable},$$
for $j=0,1,\dots,k$.

For the base case of $j=0$, the set of vectors $\{Z_1, \dots, Z_k\}$ is $(\epsilon/2)$-stable by assumption.

For the induction step, assume that 
$$\lt\{\ztilde 1, \ztilde 2, \dots, \ztilde j, Z_{j+1}, Z_{j+2},\dots Z_k\rt\}\qquad
\mbox{ is }\lt(\epsilon - \frac{j\epsilon}{2k}\rt)\mbox{-stable}.$$
By Lemma~\ref{indlemma}, we have that 
$$\lt\{\ztilde 1, \ztilde 2, \dots, \ztilde {j+1}, Z_{j+2}, Z_{j+3},\dots Z_k\rt\}$$
is $$\lt(\epsilon - \frac{j\epsilon}{2k}-f_{j+1}(n)\vnorm{\phi_{j+1}}\lt( \frac{20}\epsilon\rt) \max\lt\{\epsilon/20,\vnorm{\ztilde 1},\dots,\vnorm{\ztilde j},\vnorm{Z_{j+1}}, \dots,\vnorm{Z_{j+1}}\rt\}\rt)\mbox{-stable}.$$
By the assumed lower bound on $\epsilon$, we have that
\begin{align*}
&\hspace{-.5cm} \epsilon - \frac{j\epsilon}{2k}-f_{j+1}(n)\vnorm{\phi_{j+1}}\lt( \frac{20}\epsilon\rt) \max\lt\{\epsilon/20,\vnorm{\ztilde 1},\dots,\vnorm{\ztilde j},\vnorm{Z_{j+1}}, \dots,\vnorm{Z_{j+1}}\rt\}\\
&\ge \epsilon - \frac{j\epsilon}{2k}-\lt(\frac{20}{\epsilon}\rt)\max_{1\le i\le
k}\{f_i(n)\vnorm{\phi_{i}}\}\lt(\max_{1 \le i\le
k}\lt\{\epsilon/20,\vnorm{Z_i}\rt\}+  \max_{1\le i\le
k}\{f_i(n)\vnorm{\phi_{i}}\} \rt)\\
&\ge \epsilon - \frac{j\epsilon}{2k}-\frac{\epsilon}{2k} = \epsilon - \frac{(j+1)\epsilon}{2k}.
\end{align*}

\end{proof}

\begin{proposition}\label{stabprop}
Let $v, Z_1,\dots,Z_k$ and $\phi_1,\dots,\phi_k$ be $n$-dimensional complex vectors, let $f_1(n),\dots,f_k(n)$ be non-negative real functions of $n$, and let
$\ztilde i= Z_i+f_i(n)\phi_i$ for each $1\le i\le k$.  If $\{Z_1,\dots,Z_k\}$ is $\epsilon$-stable and
\begin{equation}\label{epscond2}
20\ge \epsilon/2 > \sqrt{40 k \max_{1\le i\le
k}\{f_i(n)\vnorm{\phi_i}\}\lt(\max_{1\le i\le k}\{1,\vnorm{Z_i}\}+\max_{1\le
i\le k} \{f_i(n)\vnorm{\phi_i}\} \rt)}.
\end{equation}
then
\begin{align*}
&\hspace{-1.5cm}\abs{
\dist( v, \Span\{Z_1,\dots,Z_{k}\})-
\dist( v, \Span\{\ztilde 1,\ztilde 2,\dots,\ztilde{k}\})
}\\
&\le \lt(\frac{40}{\epsilon}\rt) k \vnorm v \max_{1\le i \le k} \{
f_i(n)\vnorm{\phi_i}\}\\
&\le 
\vnorm{v} \sqrt{\frac{10k\max_{1\le i\le k}\{f_i(n) \vnorm{\phi_i}\} }{\max_{1\le i\le k}\{1,\vnorm{Z_i}\}} }
\end{align*}

\end{proposition}

\begin{proof}
The proposition follows from adding the perturbations one at a time (similar versions of the Lindeberg trick have been of recent use in random matrix theory, see for example \cite{Chatterjee2006}, \cite{TVuK2008}, \cite{TVu-sval2010}).
We will use Lemmas~\ref{contstab} and \ref{stablem2} to bound the successive differences, showing that the sum of the successive differences is at most the desired bound.

By Lemma~\ref{contstab}, we know that for each $0\le j\le k$ that
$$\{\ztilde1, \ztilde2, \dots, \ztilde j, Z_{i+1}, \dots ,Z_k\}$$
is $(\epsilon/2)$-stable.  Our plan is now to apply Lemma~\ref{stablem2}
repeatedly, noting that $\epsilon/2 > 2 f_i(n)\vnorm{\phi_i}$ follows
assumption \eqref{epscond2}.

Let 
\begin{align*}
h(0) &= \dist(v,\Span\{Z_1,\dots,Z_{k}\})\\
h(i) &= \dist(v,\Span\{\ztilde 1,\dots,\ztilde i, Z_{i+1}, \dots, Z_{k}\})\\
h(k) &=\dist(v,\Span\{\ztilde 1,\dots,\ztilde{k}\}).
\end{align*}
To prove the propositon we must bound $\abs{ h(0) - h(k)}$.

We note that
\begin{align*}
\abs{ h(0) - h(k)} &\le \sum_{i=1}^k \abs{h(i-1)-h(i)}\\
&\le  \sum_{i=1}^{k}\frac{40}{\epsilon} f_i(n)\vnorm{v}\vnorm{\phi_i} 
&&\mbox{\footnotesize (by Lemma~\ref{stablem2})}\\
&\le  k\lt(\frac{40}{\epsilon}\rt) \vnorm{v}\max_{1\le i\le k}\{f_i(n)\vnorm{\phi_i}\}\\
&\le 
\vnorm{v} \sqrt{\frac{10k\max_{1\le i\le k}\{f_i(n) \vnorm{\phi_i}\} }{\max_{1\le i\le k}\{1,\vnorm{Z_i}\}} }
&&\mbox{\footnotesize (using \eqref{epscond2})}\\
\end{align*}
completing the proof.
\end{proof}

\begin{corollary}
  \label{corpolynoise}
Let $v,Z_1,\dots,Z_k$ be $n$-dimensional complex vectors, let
$\phi_1,\dots,\phi_k$ be random complex vectors with mean zero and
variance 1, and let $\ztilde i= Z_i+n^{-\gamma}\phi_i$ for each $1\le i\le k$.
If 
$\{Z_1,\dots,Z_k\}$ is $\epsilon$-stable where and
$$20\ge \epsilon/2 > \sqrt{40 k n^{-\gamma}\sqrt{k\log n}\lt(\max_{1\le i\le
k}\{1,\vnorm{Z_i}\}+n^{-\gamma}\sqrt{k\log n} \rt)}.$$
then with probability at least $1-1/\log n$ we have
\begin{align*}
&\hspace{-.5cm}\abs{
\dist( v, \Span\{Z_1,\dots,Z_{k}\})-
\dist( v, \Span\{\ztilde 1,\ztilde 2,\dots,\ztilde{k}\})
}\\
&\le \vnorm{v}
\sqrt{\frac{10  k  n^{-\gamma}\sqrt{k\log n}}{\max_{1\le i\le k}\{1,\vnorm{Z_i}\}} }
\end{align*}
\end{corollary}

\begin{proof}
 Combine Proposition~\ref{stabprop} with the fact that $\max_{1\le i\le k} \vnorm{\phi_i} \le \sqrt{k \log n}$ with probability at least $1-1/\log n$ (from Chebyshev's inequality and the union bound).
\end{proof}

In applying the small noise results such as Corollary~\ref{corpolynoise} or
Proposition~\ref{corpolymat} (below) to determiting the limiting ESD,
re-ordering the rows and corresponding columns has no effect on the
eigenvalues.  This allows assumptions such as the rows being ordered by
decreasing norm to be easily met.  Recall that Proposition~\ref{corpolymatA} was a key part in proving Theorem~\ref{mainthm}.  Before proving Proposition~\ref{corpolymatA} below, will re-state the result as Proposition~\ref{corpolymat}.  For notational simplicity, we will absorbe the $zI$ term in Proposition~\ref{corpolymatA} into $M_n$ in the proposition below, which also lets us state the proposition in a self-contained way without referencing $z$.

\begin{proposition}[same statement as Proposition~\ref{corpolymatA}]
  \label{corpolymat}
\Corpolymatstmt
\end{proposition}

\newcommand{\eri}{\mathfrak{a}_i}
\newcommand{\heri}{\mathfrak{b}_i}

\begin{proof}
The main tool here is repeated application of Corollary~\ref{corpolynoise}.
Throughout, we will use the fact (as in the proof of
Corollary~\ref{corpolynoise}) that $\max_{1\le i\le m}\vnorm{\phi_i} \le
\sqrt{n\log n}$ with probability at least $1-1/\log n$.

To start, let $d_i(M_n+n^{-\gamma}\Phi)= d_i(M_n)+\eri$ and
$d_i(M_n+n^{-\gamma}\Psi) = d_i(M_n)+\heri$, where $\eri$ and $\heri$ are error
terms. We can use Corollary~\ref{corpolynoise} and
Lemma~\ref{stablem1} to bound these error terms: 
\begin{align*}\abs{\eri}, \abs{\heri} &\le
\frac{\vnorm{Z_i}}{\sqrt{\max_{1\le j \le i}\{1,\vnorm{Z_j}\}}}\sqrt{10}n^{-\gamma/2+3/4}\log^{1/4}n +
n^{-\gamma+1/2}\log^{1/2} n\\
&\le \vnorm{Z_i}^{1/2}4n^{-\gamma/2+3/4}\log^{1/4}n,
\end{align*}
where the second inequality holds for sufficiently large $n$.

We note that $$\abs{ \log d_i(M_n+n^{-\gamma}\Phi) - \log
d_i(M_n+n^{-\gamma}\Psi) } = \log\lt( 1 + \frac{\eri - \heri}{d_i(M)+
\heri}\rt),$$
and, furthermore, that the fraction $\frac{\eri - \heri}{d_i(M)+ \heri}$ tends to
zero.  In particular
\begin{align*}
\abs{\frac{\eri - \heri}{d_i(M)+ \heri}} &\le
\frac{\vnorm{Z_i}^{1/2}8n^{-\gamma/2+3/4}\log^{1/4}n}{\epsilon - \vnorm{Z_i}^{1/2}4n^{-\gamma/2+3/4}\log^{1/4}n} \\
&\le \frac{\vnorm{Z_i}^{1/2}8\delta_{n,\epsilon}}{\sqrt{\max_{1\le i \le m}\{1,
\vnorm{Z_i}\}} -
\vnorm{Z_i}^{1/2}4\delta_{n,\epsilon}}\\
&\le \frac{8\deltane}{1-4\deltane} \to 0
\end{align*}
as $n \to \infty$, by the assumption on $\deltane$.  Thus, we can use the
approximation $\log(1+x) \le 2 \abs x$, which holds for $-0.797\le x $, to write
$$\abs{ \log d_i(M_n+n^{-\gamma}\Phi) - \log
d_i(M_n+n^{-\gamma}\Psi) } = \log\lt( 1 + \frac{\eri - \heri}{d_i(M)+
\heri}\rt) \le 2 \abs{\frac{8\deltane}{1-4\deltane} } < 20\deltane,$$
for sufficiently large $n$.

Using the triangle inequality and the above approximation, we have
$$\abs{\frac1n \sum_{i=1}^{m} \log d_i(M_n+n^{-\gamma}\Phi) - \log
d_i(M_n+n^{-\gamma}\Psi)} \le \frac{m}{n} 20\deltane \le 20\deltane$$
\end{proof}

\section{Proofs of applications}
\label{S:AppPfs}

We first state two lemmas describing the $\epsilon$-stability of subsets of the rows of $T_{b,n}+zI$ and then give the proofs of Theorems~\ref{smsmallblocks} and \ref{smallblocks} in Subsection~\ref{ss:smsmallblocks} and Subsection~\ref{ss:smallblocks} below.

\begin{lemma}
 \label{Tbn-epstab}
Let $e_i$ denote the standard basis vector in $\bb C^n$ with a $1$ in position $i$ and zeros elsewhere.  Let $m \le n-1$, let $J$ be a subset of $\{1,2,\dots,m\}$, and consider the set $S=\{z e_i + e_{i+1}: i \in J\}$, where $z\in \bb C$ and $\abs z \ne 1$.   Then  
$$S\mbox{ is }\min\braces{1,\abs{1-\abs z^2}^{1/2}}\mbox{-stable}.$$
\end{lemma}

\begin{proof}[Sketch]
One procedes by finding  an orthogonal basis for $S$ and then minimizing over the distance from one vector in $S$ to the rest.  Details appear in Subsection~\ref{ss:Tbn-epstab}.
\end{proof}

\begin{lemma}
 \label{Tbn-epstab-all}
Let $S=\{\vv 1,\dots,\vv n\}$ be the set of the rows of $T_{b,n}+zI$, where $z\in\bb C$.  Then 
$$S\mbox{ is }
\begin{cases}
\abs{\abs z^2-1}^{1/2}\mbox{-stable} & \mbox{ if } \abs z > 1 \\
\abs z^{b+1}\abs{1-\abs z^2}^{1/2}\mbox{-stable} & \mbox{ if } \abs z < 1. 
\end{cases}
$$
\end{lemma}

\begin{proof}[Sketch]
The proof uses a similar approach to Lemma~\ref{Tbn-epstab}, with the change that rows of  form  $ze_i$ make orthogonalizing much simplier, resulting in parts of the orthogonal basis being equal to a rescaling of the standard basis.  Details appear in Subsection~\ref{ss:Tbn-epstab-all}.
\end{proof}

\subsection{Proof of Theorem~\ref{smsmallblocks}}\label{ss:smsmallblocks}

By Lemma~\ref{Tbn-epstab-all}, for each constant $\abs z\ne 1$, we know that the set of all rows of $T_{b,n}+zI$ is $\epsilon$-stable for a constant $\epsilon$ when $\abs z >1$, and we know that the set of all rows is $\epsilon$-stable for some $\epsilon > \Omega( n^{o(1)})$ where $o(1)\to 0$ as $n\to \infty$ when $\abs z < 1$.  Thus, by Proposition~\ref{corpolymatA}, 
$$\abs{\frac1n \sum_{i=1}^{n} \log d_i(T_{b,n}+n^{-\gamma}\Phi) - \log d_i(T_{b,n})} \le \delta_{n,\epsilon},$$
where $\delta_{n,\epsilon}\le O(n^{1.5-\gamma-o(1)}) \to 0$ as $n \to \infty$.
The above shows that that $A_n=T_{b,n}$ and $B_n=T_{b,n}+n^{-\gamma}\Phi$ satisfy condition (ii) of Theorem~\ref{TVuK2.1}, and it can also be shown (similarly to Lemma~\ref{Aasb}) that the same $A_n$ and $B_n$ satisfy condition (i) of Theorem~\ref{TVuK2.1}.  Thus, the ESD of $T_{b,n}+n^{-\gamma}\Phi$ converges in probability to the ESD of $T_{b,n}$, which is the Dirac delta $\delta_0$ with mass 1 at the origin.  \hfill $\square$

\subsection{Proof of Theorem~\ref{smallblocks}}\label{ss:smallblocks}

First we show that the ESD of $T_{b,n}+n^{-\gamma}\Phi_n$ is the same as the ESD of $T_n+n^{-\gamma}\Phi_n$.  Note that there are less than $n/(b +1) = o(n/\log n)$ rows of $T_{b,n}$ that contain all zeros.  Thus, there are at most $o(n/\log n)$ rows of $T_{b,n}+n^{-\gamma}\Phi_n$ that differ from the corresponding rows of $T_{n}+n^{-\gamma}\Phi_n$.  Combining Theorem~\ref{TVuK2.1}, Proposition~\ref{propdist}, and Lemma~\ref{lem:Hdim} (re-ordering rows and columns of the matrices so that the rows that differ are the last rows), we see that the difference of the ESDs of $T_{b,n}+n^{-\gamma}\Phi_n$ and $T_{n}+n^{-\gamma}\Phi_n$ converges to zero in probability.

Second, we will show that the ESD of $T_{n}+n^{-\gamma}\Phi_n$ is the same as the ESD of $T_{n}+n^{-\gamma}G_n$, where $\sqrt n G_n$ is an iid Ginibre matrix, so each entry of $\sqrt n G_n$ is complex Gaussian with mean zero and variance one.  Here we apply Theorem~\ref{mainthm}, noting that the set of the first $n-1$ rows is $\epsilon$-stable for a constant $\epsilon$ (where $\epsilon$ depends on $z$), thus proving that the difference of the ESDS of $T_{n}+n^{-\gamma}\Phi_n$ and $T_{n}+n^{-\gamma}G_n$ converges to zero in probability.

Finally, it was proved in \cite{GWZeitouni2011} (see also \cite{Sniady2002}) that the ESD of $T_{b,n}+n^{-\gamma}G_n$ converges in probability to uniform on the unit circle, completing the proof. \hfill$\square$

\subsection{Proof of Lemma~\ref{Tbn-epstab}}
\label{ss:Tbn-epstab}

\newcommand\epz{\epsilon_{z,m}}

Note that if $S$ is a subset of $T$ and $T$ is $\epsilon$-stable, then $S$ is also $\epsilon$-stable.  Thus, it is sufficient to show that $S=\{ze_i+ e_{i+1}: 1 \le i \le m\}$ is $\epz$-stable, where $\epz=\min\{1,\abs{ 1 - \abs z ^2}^{1/2}\}$.

Fix $\ell \in \{1,2,\dots,m\}$.  We need to show that
\begin{equation}\label{thedist}
\dist( ze_\ell + e_{\ell+1}, \Span\{ ze_i+e_{i+1}: 1\le i \le m \mbox{ and } i\ne \ell\})\ge \epz.
\end{equation}
We will find orthogonal bases for $\{ ze_i+e_{i+1}: 1\le i \le \ell-1\}$ and for $\{ ze_i+e_{i+1}: \ell+1\le i \le m\}$, noting that together they form an orthogonal basis for $\{ ze_i+e_{i+1}: 1\le i \le m \mbox{ and } i\ne \ell\}$.  Then we will use the orthogonal basis to compute the distance in \eqref{thedist} explicitly.

\begin{lemma}\label{topdown}
 The vectors $\{ ze_i + e_{i+1}:1\le i \le \ell-1\}$ where $e_i$ is the $i$-th standard basis vector, have an orthogonal basis $\{w_1,\dots, w_{\ell-1}\}$ where 
\begin{align*}
w_k &= e_{k+1} + \frac{z\lt(\abs z^{2(k-1)}e_k +(-z) \abs z^{2(k-2)}e_{k-1}+(-z)^2 \abs z^{2(k-3)}e_{k-2} +\dots+(-z)^{k-1} e_1  \rt)}{1+\abs z^2+\abs z^4+ \dots +\abs z^{2k-2}}  \\
&= e_{k+1} + \frac{z}{\sum_{i=0}^{k-1} \abs z ^{2i}} \lt( \sum_{j=0}^{k-1} (-z)^{k-1-j} \abs z^{2j} e_{j+1}\rt).
\end{align*}
\end{lemma}

\begin{proof}
 We proceed by induction on $\ell$.  For the base case $\ell =2$, we have that $w_1 = e_2 + z e_1$, as it should.

For the induction step, assume the result for $\ell -1$ where $\ell \ge 3$, which gives an orthogonal basis $\{w_1,\dots, w_{\ell -2}\}$ with the form above for the set $\{ze_i+e_{i+1}: 1\le i\le \ell-2\}$. We will now orthogonalize $ze_{\ell-1} + e_{\ell}$ with respect to $\{w_1,\dots, w_{\ell-2}\}$, showing that the resulting vector equals $w_{\ell-1}$ with the form above.

To orthogonalize $v:=ze_{\ell-1} + e_{\ell}$ we compute $v - \frac{w_{\ell-2}}{\vnorm{w_{\ell-2}}} z$ (since $\angles{v,w_i} = 0$ for $1\le i\le \ell-3$ and $\angles{v,w_{\ell-2}}=z$).  Note that
\begin{align}
\vnorm{w_{\ell-2}}^2
&=  1 + \frac{\abs z^2}{\lt(\sum_{i=0}^{\ell-3} \abs z^{2i}\rt)^2} \sum_{j=0}^{\ell-3} \abs z^{2(\ell-2) -2j -2} \abs z^{4j} \nonumber\\
&= 1 + \frac{\abs z^{2(\ell -2)} \sum_{j=0}^{\ell-3} \abs z^{2j}}{\lt(\sum_{i=0}^{\ell-3} \abs z^{2i} \rt)^2} 
= 1 +\frac{\abs z^{2(\ell -2)}}{\sum_{i=0}^{\ell-3} \abs z^{2i}}\nonumber\\
&= \frac{\sum_{j=0}^{\ell-2} \abs z ^{2j}}{\sum_{i=0}^{\ell-3} \abs z^{2i}}. \label{topdownlength}
\end{align}
Thus the orthogonalization of $v$ with respect to $\{w_1,\dots, w_{\ell-2}\}$ is
\begin{align*}
w_{\ell-1}&:=e_\ell + z e_{\ell -1} - z \frac{\sum_{i=0}^{\ell-3} \abs z^{2i}}{\sum_{j=0}^{\ell-2} \abs z ^{2j}}\lt( e_{\ell-1} + \frac{z}{\sum_{i=0}^{\ell-3} \abs z ^{2i}} \lt( \sum_{j=0}^{\ell-3} (-z)^{\ell-2-1-j} \abs z^{2j} e_{j+1}\rt),
\rt)\\
&= e_\ell+ z e_{\ell -1} \lt(1 -\frac{\sum_{i=0}^{\ell-3} \abs z^{2i}}{\sum_{j=0}^{\ell-2} \abs z ^{2j}}\rt) +  \frac{z}{\sum_{i=0}^{\ell-2} \abs z ^{2i}} \lt( \sum_{j=0}^{\ell-3} (-z)^{\ell-2-j} \abs z^{2j} e_{j+1}\rt)\\
&= e_\ell+ \lt(\frac{z}{\sum_{j=0}^{\ell-2} \abs z ^{2j}}\rt) \abs z^{2(\ell-2)}e_{\ell -1}  +  \frac{z}{\sum_{i=0}^{\ell-2} \abs z ^{2i}} \lt( \sum_{j=0}^{\ell-3} (-z)^{\ell-2-j} \abs z^{2j} e_{j+1}\rt)\\
&= e_\ell+  \frac{z}{\sum_{i=0}^{\ell-2} \abs z ^{2i}} \lt( \sum_{j=0}^{\ell-2} (-z)^{\ell-2-j} \abs z^{2j} e_{j+1}\rt).
\end{align*}
This is the desired form for $w_{\ell-1}$, completing the proof by induction.
\end{proof}

\begin{lemma}\label{botup}
The vectors $\{ ze_i + e_{i+1}:\ell+1\le i \le m\}$ where $e_i$ is the $i$-th standard basis vector, have an orthogonal basis $\{w_{\ell+1},\dots, w_{m}\}$ where 
\begin{align*}
w_k &= z e_k + \frac{e_{k+1} -\zbar e_{k+2} + \zbar^2 e_{k+3} + \dots + (-\zbar)^{m-k} e_{m+1}}{1+\abs z^2 +\abs z^4+ \dots+z^{2m-2k}}\\
&= z e_k + \frac{1}{\sum_{i=0}^{m-k} \abs z^{2i}}\lt(\sum_{j=k+1}^{m+1} (-\zbar)^{j-k-1}e_j \rt)
\end{align*}
\end{lemma}

\begin{proof}
We proceed by induction on $\ell$.  For the base case of $\ell = m-1$, there is only one vector, thus $w_{\ell+1}=w_{m} = ze_m + e_{m+1}$, as it should.

For the induction step, we will assume the result for $\ell+2$ and show that it must also hold for $\ell +1$.  We will orthogonalize $ze_{\ell+1} + e_{\ell +2}$ with respect to $\{w_{\ell+2},\dots,w_{m}\}$, assuming $w_{\ell+2},\dots, w_m$ have the form above.  Thus, we have
\begin{align*}
 w_{\ell+1} := ze_{\ell+1} + e_{\ell +2} - \angles{ze_{\ell+1} + e_{\ell +2},w_{\ell+2}} \frac{w_{\ell+2}}{\vnorm{w_{\ell+2}}^2}.
\end{align*}
Note that 
\begin{equation}\label{botuplength}
 \vnorm{w_{\ell+2}}^2= \abs z^2 + \frac{\sum_{j=\ell+3}^{m+1} \abs z^{2j-2(\ell +2)-2}}{\lt(\sum_{i=0}^{m-(\ell+2)} \abs z ^{2i}\rt)^2} = \abs z^2 + \frac{1}{\sum_{i=0}^{m-(\ell+2)} \abs z ^{2i}} = \frac{\sum_{i=0}^{m-(\ell+1)} \abs z ^{2i}}{\sum_{i=0}^{m-(\ell+2)} \abs z ^{2i}}.
\end{equation}
Thus
\begin{align*}
 w_{\ell+1} &= ze_{\ell+1} + e_{\ell +2} - \zbar \frac{\sum_{i=0}^{m-(\ell+2)} \abs z ^{2i}}{\sum_{i=0}^{m-(\ell+1)} \abs z ^{2i}}\lt(z e_{\ell+2} + \frac{1}{\sum_{i=0}^{m-(\ell+2)} \abs z^{2i}}\lt(\sum_{j=\ell+2+1}^{m+1} (-\zbar)^{j-(\ell+2)-1}e_j \rt) \rt)\\
&= ze_{\ell+1} + e_{\ell+2} \frac{1}{\sum_{i=0}^{m-(\ell+1)} \abs z ^{2i}} +
\frac{-\zbar}{\sum_{i=0}^{m-(\ell+1)} \abs z ^{2i}}\lt(\sum_{j=\ell+2+1}^{m+1} (-\zbar)^{j-(\ell+2)-1}e_j \rt)\\
&= ze_{\ell+1} +\frac{1}{\sum_{i=0}^{m-(\ell+1)} \abs z ^{2i}}\lt(\sum_{j=\ell+2}^{m+1} (-\zbar)^{j-(\ell+2)}e_j \rt).
\end{align*}
Thus $w_{\ell+1}$ has the desired form, completing the proof by induction. 
\end{proof}

We will now use Lemmas~\ref{topdown} and \ref{botup} to explicitly compute the distance on the left side of \eqref{thedist}, which will lead to a proof that $\{ze_i+e_{i+1}: 1 \le i \le m\}$ is $(\epsilon_{z,m})$-stable.  We will consider 3 cases: where $\ell=1$, where $2\le \ell \le m-1$, and where $\ell = m$.

For the $\ell=1$ case, the distance from $ze_1+e_2$ to $\Span\{ze_i+e_{i+1}: 2 \le i \le m\}$ is the length of $w_1$ using Lemma~\ref{botup}, which is (see \eqref{botuplength})
$$
\vnorm{w_1}=\lt(\frac{\sum_{i=0}^{m} \abs z ^{2i}}{\sum_{i=0}^{m-1} \abs z ^{2i}} \rt)^{1/2} = \lt( \frac{1-\abs z^{2m +2}}{1-\abs z^{2m}} \rt)^{1/2},
$$
assuming $\abs z \ne 1$.  When $\abs z < 1$, we have $\vnorm{w_1} \ge 1$; and when $\abs z > 1$, we have $\vnorm{w_1}\ge (\abs z^{2m+2}/\abs z^{2m})^{1/2} = \abs z> 1$.  Thus, assuming $\abs z \ne 1$, the distance on the left side of \eqref{thedist} is at least $1$ when $\ell =1$.

For the $\ell = m$ case, the distance from $ze_m+e_{m+1} $ to $\Span\{ze_i+e_{i+1}: 1 \le i \le m-1\}$ is the length of $w_m$ using Lemma~\ref{topdown}, which is (see \eqref{topdownlength})
$$
\vnorm {w_m} =\lt(\frac{\sum_{j=0}^{m} \abs z ^{2j}}{\sum_{i=0}^{m-1} \abs z^{2i}}\rt)^{1/2}
= \lt(\frac{1-\abs z^{2m +2}}{1-\abs z^{2m}}\rt)^{1/2},
$$
assuming $\abs z \ne 1$.  When $\abs z < 1$, we have $\vnorm{w_m} \ge 1$; and when $\abs z > 1$, we have $\vnorm{w_m} \ge (\abs z^{2m+2}/\abs z^{2m})^{1/2} = \abs z > 1$. Thus, assuming $\abs z \ne 1$, the distance on the left side of \eqref{thedist} is at least $1$ when $\ell =m$.

For the $2\le \ell \le m-1$ case, the distance from $ze_\ell + e_{\ell+1}$ to $\Span\{ze_i+e_{i+1}: 1\le i \le m, i \ne \ell \}$ is more complicated.  The orthogonal basis for $\{ze_i+e_{i+1}: 1\le i \le m, i \ne \ell \}$ is $\{w_1,\dots,w_{\ell-1}\} \cup \{w_{\ell+1},\dots, w_m\}$, where the first $\ell-1$ vectors are orthogonalized using Lemma\ref{topdown} and the last $m-\ell$ vectors are orthogonalized using Lemma~\ref{botup}.  
The distance in question is equal to the norm of $v$ where
\begin{align*}
v&:=ze_{\ell} + e_{\ell+1} - \frac{z w_{\ell-1}}{\vnorm{w_{\ell-1}} ^2} - \frac{\zbar w_{\ell+1}}{\vnorm{w_{\ell+1}}^2}\\
&= z e_\ell - z  \frac{\sum_{i=0}^{\ell-2} \abs z^{2i}}{\sum_{j=0}^{\ell-1} \abs z ^{2j}}\lt( e_{\ell} + \frac{z}{\sum_{i=0}^{\ell-2} \abs z ^{2i}} \lt( \sum_{j=0}^{\ell-2} (-z)^{\ell-2-j} \abs z^{2j} e_{j+1}\rt)\rt) \\
&\qquad + e_{\ell+1} - \zbar \frac{\sum_{i=0}^{m-(\ell+1)} \abs z ^{2i}}{\sum_{i=0}^{m-\ell} \abs z ^{2i}}\lt(z e_{\ell+1} + \frac{1}{\sum_{i=0}^{m-(\ell+1)} \abs z^{2i}}\lt(\sum_{j=\ell+2}^{m+1} (-\zbar)^{j-(\ell+1)-1}e_j \rt) \rt)\\
&=   \frac{z}{\sum_{i=0}^{\ell-1} \abs z ^{2i}} \lt( \sum_{j=0}^{\ell-1} (-z)^{\ell-1-j} \abs z^{2j} e_{j+1}\rt) +\frac{1}{\sum_{i=0}^{m-\ell} \abs z ^{2i}}\lt(\sum_{j=\ell+1}^{m+1} (-\zbar)^{j-(\ell+1)}e_j \rt).
\end{align*}
The above vector $v$ has norm-squared
\begin{align*}
\vnorm v^2 &= \frac{\abs z^{2\ell}}{\sum_{i=0}^{\ell-1} \abs z^{2i}} + 
\frac{1}{\sum_{i=0}^{m-\ell} \abs z ^{2i}}\\ 
&= \frac{\abs z^{2\ell}(1-\abs z^2)}{1-\abs z^{2\ell}} + \frac{1-\abs z^2}{1-z^{2m-2\ell+2}}\\
&= (1-\abs z^2)\lt( -1 + \frac1{1-\abs z^{2\ell}} + \frac 1{1-\abs{z}^{2m-2\ell+2}}\rt)
\end{align*}
assuming $\abs z \ne 1$.
If $\abs z > 1$, then $\vnorm v \ge (\abs z^2 -1)^{1/2}$; and if $\abs z < 1$, then $\vnorm v \ge (1-\abs z^2)^{1/2}$.  Thus, assuming $\abs z \ne 1$, the distance on the left side of \eqref{thedist} is at least $\abs{1-\abs z^2}^{1/2}$ when $2 \le \ell \le m-1$.

Putting together the three cases above, we have proven \eqref{thedist} with $\epsilon_{z,m} = \min\{1,\abs{1-\abs z^2}^{1/2}\}$, completing the proof of Lemma~\ref{Tbn-epstab}. \hfill $\square$

\subsection{Proof of Lemma~\ref{Tbn-epstab-all}}
\label{ss:Tbn-epstab-all}

Recall that $T_{b,n} + z I$ is a block diagonal matrix in which $b+1$ by $b+1$ each block has the form
$$\lt(\begin{matrix}
z & 1 & 0 & \dots&  0 \\
0 & z & 1 & 0 &  \vdots\\
\vdots &&\ddots & \ddots& \vdots \\
0  &\dots&&z & 1 \\
0 & 0 & \dots & 0 & z 
\end{matrix}\rt).$$
If $b+1$ does not divide $n$ evenly, the last block is a smaller $k$ by $k$ block (where $k \le b$) also having the form above.   The blocks are orthogonal, so to compute the distance from a given row to the span of the other rows in $T_{b,n}+zI$, it is sufficient to compute the distance from the given row to the span of the other rows in the same block.  Thus, we will show that 
\begin{equation}\label{newdist}
\dist(z e_{b+1}, \Span\{ze_i + e_{i+1}: 1 \le i \le b\}) \ge \epsilon_{z, b}
\end{equation}
and that
\begin{equation}\label{thedist2}
\dist(z e_{\ell} + e_{\ell+1}, \Span\lt (\{ze_i + e_{i+1}: 1 \le i \le b, i \ne \ell\} \cup \{ ze_{b+1}\} \rt)) \ge \epsilon_{z, b},
\end{equation}
where $\epsilon_{z,m} \ge \abs { \abs z^2 -1}^{1/2}$ if $\abs z > 1$ and $\epsilon_{z,m} \ge \abs z^{b+1}\abs{1-\abs z^2}^{1/2}$ if $\abs z < 1$. 

To prove \eqref{newdist}, we orthogonalize $\{ze_i + e_{i+1}: 1 \le i \le b\}$ the basis $\{w_1,\dots,w_{b}\}$ with the form in Lemma~\ref{topdown} (letting $\ell=b+1$).  The distance from $ze_{b+1}$ to $\Span\{ze_i + e_{i+1}: 1 \le i \le b\}$ is thus the length of the vector 
\begin{align*}
v&:=ze_{b+1} - \angles{z e_{b+1}, w_b} \frac{w_b}{\vnorm{w_b}^2} \\
&= ze_{b+1} - z 
\frac{\sum_{i=0}^{b-1} \abs z^{2i}}{\sum_{j=0}^{b} \abs z ^{2j}}
\lt(
e_{b+1} + \frac{z}{\sum_{i=0}^{b-1} \abs z ^{2i}} \lt( \sum_{j=0}^{b-1} (-z)^{b-1-j} \abs z^{2j} e_{j+1}\rt)
\rt)\\
&=
\frac{z}{\sum_{i=0}^{b} \abs z ^{2i}} \lt( \sum_{j=0}^{b} (-z)^{b-j} \abs z^{2j} e_{j+1}\rt).
\end{align*}
Thus we have 
$$\vnorm v ^2= \frac{\abs z^{2(b+1)}}{\sum_{i=0}^{b} \abs z^{2i}} = \abs z^{2(b+1)} \frac{1-\abs z^2}{1-\abs z^{2(b+1)}},
$$
assuming $\abs z \ne 1$.
If $\abs z > 1$, then $\vnorm v \ge \abs { \abs z^2 -1}^{1/2}$; and if $\abs z < 1$, then $\vnorm v \ge \abs z^{b+1}\abs{1-\abs z^2}^{1/2}$.

To prove \eqref{thedist2}, we will orthogonalize $\{ze_i + e_{i+1}: 1 \le i \le b, i \ne \ell\} \cup \{ ze_{b+1}\}$ in two parts. The set $\{ze_i+ e_{i+1}: = \le i \le \ell -1\}$ has an orthogonal basis $\{w_1,\dots,w_{\ell-1}\}$ with the form in Lemma~\ref{topdown}, and, as we will show below, the remaining vectors have an orthogonal basis that is as re-scaling of the standard basis.

\begin{lemma}\label{botstan}
The vectors $\{ze_i + e_{i+1}: \ell+1 \le i \le b\} \cup \{ ze_{b+1}\}$, where $e_i$ is the $i$-th standard basis vector, have an orthogonal basis $\{ze_i: \ell+1 \le i \le b+1\}$, which is a re-scaling of the standard basis.
\end{lemma}

\begin{proof}
Let the orthogonal basis be $w_{1}, \dots, w_{b+1}$.  We orthogonalize starting with the vector $w_{b+1} = ze_{b+1}$.  Let $k< b+1$ be and integer, and assume by induction that $w_{j} = ze_{j}$ for $k+1\le j \le b+1$.  
Then 
$$w_k = ze_k + e_{k+1} - \angles{ze_k + e_{k+1}, w_{k+1}}\frac{w_{k+1}}{\vnorm{w_{k+1}}^2}
=ze_k,$$ 
completing the proof by induction.
\end{proof}

We will now compute the distance on the left side of \eqref{thedist2} explicitly using the orthogonal basis $\{w_1,\dots,w_{\ell-1}, ze_{\ell+1}, ze_{\ell+2},\dots,ze_{b+1}\}$, where the $w_i$ have the form described in Lemma~\ref{topdown}.  The distance is the length of the vector $v$ where
\begin{align*}
 v &= ze_\ell+e_{\ell+1} - \angles{ze_\ell+ e_{\ell+1},ze_{\ell+1}}\frac{ze_{\ell+1}}{\abs z^2} - \angles{ze_\ell+e_{\ell+1},w_{\ell-1}}\frac{w_{\ell-1}}{\vnorm{w_{\ell-1}}^2} \\
&=ze_\ell- \angles{ze_\ell,w_{\ell-1}}\frac{w_{\ell-1}}{\vnorm{w_{\ell-1}}^2} \\
&=\frac{z}{\sum_{i=0}^{\ell-1} \abs z ^{2i}} \lt( \sum_{j=0}^{b} (-z)^{\ell-1-j} \abs z^{2j} e_{j+1}\rt).
\end{align*}
Thus we have 
$$\vnorm v ^2= \frac{\abs z^{2\ell}}{\sum_{i=0}^{b} \abs z^{2i}} = \abs z^{2\ell} \frac{1-\abs z^2}{1-\abs z^{2\ell}},
$$
assuming $\abs z \ne 1$.
If $\abs z > 1$, then $\vnorm v \ge \abs { \abs z^2 -1}^{1/2}$; and if $\abs z < 1$, then $\vnorm v \ge \abs z^{\ell}\abs{1-\abs z^2}^{1/2}$.  Since $\ell < b$ by assumption, we have proved \eqref{thedist2}.

Finally, note that in case where $b+1$ does not evenly divide $n$, there is a last diagonal block in $T_{b,n} + zI$ equal to $T_{k}+zI_k$ where $k\le b$.  The arguments above apply to this block as well, with $b+1$ being replaced by $k$, and we need only note that the final lower bounds on $\epsilon_{z,k-1}$ are the same when $\abs z \ge 1$ and slightly better when $\abs z< 1$ than the corresponding bounds on $\epsilon_{z,b}$.  This completes the proof of Lemma~\ref{Tbn-epstab-all}. \hfill $\square$

\section*{Acknowledegments}
I would like to thank Ofer Zeitouni and Alice Guionnet for useful conversations related to this paper, which grew out of our work on \cite{GWZeitouni2011}, which in turn grew out of a very nice conference on random matrices at the American Institute of Mathematics in December 2010.  

\bibliographystyle{abbrv}
\bibliography{fullbibpub,fullbibUNpub}

\end{document}